\documentclass[11pt]{article}

\usepackage{epsfig}
\usepackage{amssymb, latexsym}
\usepackage{amscd}
\usepackage[all]{xy}
\usepackage{epsf}
\usepackage[mathscr]{eucal}

\usepackage{amsmath,amsfonts,amscd,amssymb,amsthm}

\usepackage[titletoc]{appendix}
\usepackage[sc]{titlesec}

\long\def\comment#1\endcomment{}

\theoremstyle{plain}

\newtheorem{theorem}{\sc Theorem}[section]
\newtheorem{lemma}[theorem]{\sc Lemma}

\theoremstyle{plain}

\theoremstyle{exercise}

\makeatletter \@addtoreset{equation}{section} \makeatother

\def\eqref#1{\thetag{\ref{#1}}}

\let\latexref=\ref
\def\ref#1{{\normalfont{\latexref{#1}}}}

\newcommand{\ldot}{{\:\raisebox{2,3pt}{\text{\circle*{1.5}}}}}
\newcommand{\udot}{{\:\raisebox{3pt}{\text{\circle*{1.5}}}}}
\def\dlim_#1{{\displaystyle\lim_{#1}}^\hdot}

\newcommand{\id}{\operatorname{\rm id}}

\newcommand{\Ob}{\mathrm{Ob}}

\renewcommand{\Bar}{\mathrm{Bar}}

\newcommand{\Hom}{\mathrm{Hom}}

\newcommand{\Hoch}{\mathrm{Hoch}}

\newcommand{\dg}{{dg}}

\newcommand{\Id}{\mathrm{Id}}

\newcommand{\Cat}{{\mathscr{C}at}}

\newcommand{\Fun}{{\mathrm{Fun}}}

\renewcommand{\k}{\Bbbk}

\newcommand{\coh}{\mathrm{coh}}
\newcommand{\Coh}{{\mathscr{C}{oh}}}

\newcommand{\sotimes}{\overset{\sim}{\otimes}}

\newcommand{\Hot}{\mathrm{Hot}}

\newcommand{\red}{\mathrm{red}}

\title{\sc{A proof of the contractibility of the 2-operad \\ defined via the twisted tensor product}}

\author{\sc{Boris Shoikhet}}
\date{}

\begin{document}\maketitle
{\footnotesize
\begin{center}{\parbox{4,5in}{{\sc Abstract.}
In our recent papers [Sh1,2], we introduced a {\it twisted tensor product} of dg categories, and provided, in terms of it, {\it a contractible 2-operad $\mathcal{O}$}, acting on the category of small dg categories, in which the ``natural transformations'' are derived. We made use of some homotopy theory developed in [To] to prove the contractibility of the 2-operad $\mathcal{O}$. The contractibility is an important issue, in vein of the theory of Batanin [Ba1,2], according to which an action of a contractible $n$-operad on $C$ makes $C$ a weak $n$-category.

In this short note, we provide a new elementary proof of the contractibility of the 2-operad $\mathcal{O}$. The proof is based on a direct computation, and is independent from the homotopy theory of dg categories (in particular, it is independent from [To] and from Theorem 2.4 of [Sh1]).
}}
\end{center}
}

\section{\sc Introduction}
\subsection{}
Let $\k$ be a field. We denote by $\Cat_\dg(\k)$ the category of small dg categories over $\k$. Its external Hom-set, denoted by $\Hom(-,-)$, is the set of dg functors.

The {\it twisted tensor product} $C\sotimes D$ of small dg categories over $\k$ was introduced in [Sh1]. It was designed to fulfill the adjunction in $\Cat_\dg(\k)$
\begin{equation}
\Hom(C\sotimes D,E)\simeq \Hom(C,\Coh(D,E))
\end{equation}
where $\Coh(D,E)$\footnote{It was denoted by $\Coh_\dg(D,E)$ in [Sh1].} denotes the dg category whose objects are dg functors $f\colon D\to E$, and whose complexes of morphisms are reduced Hochshild cochains on $D$ with coefficients in $D$-bimodule $E(f(-),g(-))$.
$$
\Coh(D,E)(f,g)=\Hoch^\udot_\red(D,E(f(-),g(-)))
$$
We recall the construction of $C\sotimes D$ in Section \ref{section1}.

It is well-known [Fa] that $\Coh(D,E)$ has the homotopy type of internal derived Hom in the homotopy category of dg categories.
We regard closed degree 0 morphisms in $\Coh(C,D)(f,g)$ as ``derived natural transformations'' $f\Rightarrow g$. The ``vertical'' composition $\circ_1$ of such derived natural transformations is the composition in $\Coh(C,D)$. There are two candidates for their ``horizontal'' composition $\circ_0$, which are homotopic (see e.g. [Sh2], Fig.2,3,4). Either of two horizontal compositions gives rise to an associative product. At the same time, the compatibility 
$$(\alpha_1\circ_1\alpha_2)\circ_0(\beta_1\circ_1\beta_2)=(\alpha_1\circ_0\beta_1)\circ_1(\alpha_2\circ_0\beta_2)$$ {\it fails} for either of them. However, it holds up to homotopy, and an action of a contractible 2-operad provides a coherent system of higher homotopies. 

We denote by $\Cat_\dg^\coh(\k)$ the category $\Cat_\dg(\k)$ with the pre-2-category structure given by $\Coh(C,D)(f,g)$ (which means that the strict underlying 1-category $\Cat_\dg(\k)$ is fixed, 2-morphisms are defined, but their compositions are not defined yet).

A contractible 2-operad, acting on $\Cat_\dg^\coh(\k)$, was firstly constructed in [Tam]. The construction was somewhat common in spirit with the solution of the classical Deligne conjecture given by McClure-Smith [MS].

Another contractible 2-operad, solving the same problem, was constructed in [Sh2]. Our idea was that, although on the level of the homotopy category $\Hot(\k)$ of $\Cat_\dg(\k)$ the internal tensor product is isomorphic to the classical one $C\otimes D$, to know it on the level of unlocalised category can be beneficial for depicting non-linear structures such as a 2-operad acting on $\Cat_\dg^\coh(\k)$. On the level of $Cat_\dg(\k)$ the corresponding ``derived tensor product'' $C\sotimes D$, the twisted tensor product, is distinct from $C\otimes D$. In particular, it is not symmetric: $C\sotimes D\not\simeq D\sotimes C$.

For a 2-ordinal $(n_1,n_2,\dots,n_k)$ we define an element of $C(\k)$
\begin{equation}
\mathcal{O}(n_1,\dots,n_k)=I_{n_k}\sotimes(I_{n_{k-1}}\sotimes(\dots\sotimes (I_{n_2}\sotimes I_{n_1})\dots))(\min,\max)
\end{equation}
Here $I_n$ is the interval category with objects $0,1,\dots,n$ and with morphisms defined as $I_n(i,j)=\k$ for $i\le j$ and $I_n(i,j)=0$ for $i>j$, and the composition given by the product in $\k$. Thus, $I_n$ is a dg category all whose morphism are closed and have degree 0. The objects $\min$ and $\max$ are $(0,0,\dots,0)$ and $(n_k,n_{k-1},\dots,n_1)$, correspondingly. 

Our result in [Sh2] states that there is a 2-operad structure [Ba1,2] on the collection of complexes of $\k$-vector spaces $\mathcal{O}(n_1,\dots,n_k)$, and that this 2-operad naturally acts on $\Cat_\dg^\coh(\k)$. In particular, $(I_1\sotimes I_1)(\min,\max)$, corresponded to the horizontal product(s) $\circ_0$, is the complex
$$
0\to \underset{\deg=-1}{\k}\xrightarrow{(1,-1)}\underset{\deg=0}{\k\oplus\k}\to 0
$$
which is interpreted as two candidates for $\circ_0$ and a homotopy between them.

\subsection{}
A 2-operad of complexes $\mathcal{O}$ is called {\it contractible} if for any 2-ordinal $(n_1,\dots,n_k)$ there is a quasi-isomorphism of complexes $\mathcal{O}(n_1,\dots,n_k)\to \k[0]$, which gives rise to a map of 2-operads $\mathcal{O}\to\underline{\k}$. Here $\underline{\k}$ is the constant 2-operad (an algebra over $\underline{\k}$ is a strict 2-category).

According to [Ba1], an action of a contractible 2-operad on a pre-2-category makes it a {\it weak $2$-category}. Thus, contractibility of our 2-operad $\mathcal{O}$, acting on $\Cat_\dg^\coh(\k)$, becomes an important issue. Namely, it makes $\Cat_\dg(\k)$ with complexes of 2-morphisms $f\Rightarrow g\colon C\to D$ defined as $\Coh_\dg(C,D)(f,g)$, a ``weak 2-category''. It implies, in particular [Ba3], that $\Coh_\dg(C,C)(\id,\id)=\Hoch^\dg(C)$ is a $C_\ldot(E_2,\k)$-algebra (the latter is the statement of the ``classical Deligne conjecture''). 

In this paper, we give a new elementary proof of the following statement:
\begin{theorem}\label{theorem1}
The 2-operad $\mathcal{O}$ is contractible. 
\end{theorem}

A proof of Theorem \ref{theorem1}, given in [Sh2], uses some homotopy theory.  It is based on the following result proven in [Sh1, Theorem 2.4]:

\begin{theorem}\label{theorem2}
Let $C,D$ be two cofibrant dg categories (for the Tabuada model structure [Tab]). Than the natural projection $C\sotimes D\to C\otimes D$ is a quasi-equivalence of dg categories. 
\end{theorem}

Theorem \ref{theorem1} easily follows from Theorem \ref{theorem2}, because the categories $I_n$ are cofibrant, and the twisted tensor product of two cofibrant dg categories is cofibrant again (by [Sh1, Lemma 4.5]).

Note that the proof of Theorem \ref{theorem2} relies on some deep results in homotopy theory of $\Cat_\dg(\k)$, in particular, on [To].

\subsection{\sc}
In this paper, we give a short elementary proof of Theorem \ref{theorem1}, independent from any homotopy theory. 

We prove the following statement (which is complementary to Theorem \ref{theorem2}):
\begin{theorem}\label{theorem3}
Let $C$ be any (not necessarily cofibrant) dg category. Then the natural projection $I_n\sotimes C\to I_n\otimes C$, $n\ge 0$, is a quasi-equivalence. 
\end{theorem}

Theorem \ref{theorem1} can be easily deduced from Theorem \ref{theorem3}. At the same time, our elementary reasoning seemingly can not be upgraded to a proof of Theorem 2. 

\vspace{2mm}

We recall the definition and basic properties of the twisted tensor product in Section \ref{section1}, and prove Theorem \ref{theorem3} and then Theorem \ref{theorem1} in Section \ref{section2}.

\subsection*{}
\subsubsection*{\sc Acknowledgements} 
The work was partially supported by the FWO Research Project Nr. G060118N
and by the Russian Academic Excellence Project ‘5-100’.

\section{\sc Reminder on the twisted tensor product}\label{section1}
Here we recall, for reader's convenience, the definition and some basic facts on the twisted tensor product from [Sh1]. 

\subsection{\sc The definition}\label{sectiondeftwist}
Let $C$ and $D$ be two small dg categories over $\k$. We define {\it the twisted dg tensor product} $C\sotimes D$, as follows.

The set of objects of $C\sotimes D$ is $\Ob(C)\times \Ob(D)$. Consider the graded $\k$-linear category $F(C,D)$ with objects $\Ob(C)\times \Ob(D)$, freely generated by the morphisms in $\{C(c_1,c_2)\otimes\id_d\}_{ c_1,c_2\in C, d\in D}$, $\{\id_c\otimes D(d_1,d_2)\}_{c\in C, d_1,d_2\in D}$, and by the new morphisms $\varepsilon(f;g_1,\dots,g_n)$, specified below.

For
$$
c_0\xrightarrow{f}c_1\text{  and  }d_0\xrightarrow{g_1}d_1\xrightarrow{g_2}\dots\xrightarrow{g_n}d_n
$$
chains of composable maps in $C$ and in $D$, correspondingly, with $n\ge 1$, one introduces a morphism
$$
\varepsilon(f;g_1,\dots,g_n)\in \Hom(c_0\times d_0,c_1\times d_n)
$$
of degree
\begin{equation}
\deg \varepsilon(f;g_1,\dots,g_n)=-n+\deg f_1+\sum\deg g_j
\end{equation}

The underlying graded category of $C\sotimes D$ is defined as the quotient of $F(C,D)$ by the two-sided ideal, defined by the following identities:

\begin{itemize}
\item[$(R_1)$]
$(\id_c\otimes g_1)* (\id_c\otimes g_2)=\id_c\otimes (g_1g_2)$, $(f_1\otimes\id_d)*(f_2\id_d)=(f_1f_2)*\id_d$
\item[$(R_2)$] $\varepsilon(f;g_1,\dots,g_n)$ is linear in each argument,
\item[$(R_3)$] 
$\varepsilon(f; g_1,\dots,g_n)=0$ if $g_i=\id_y$ for some $y\in \Ob(D)$ and for some $1\le i\le n$,\\
$\varepsilon(\id_x; g_1,\dots,g_n)=0$ for $x\in\Ob(C)$ and $n\ge 1$,
\item[$(R_4)$]
for any $c_0\xrightarrow{f_1}c_1\xrightarrow{f_2}c_2$ and $d_0\xrightarrow{g_1}d_1\xrightarrow{g_2}\dots\xrightarrow{g_N}d_N$
one has:
\begin{equation}\label{eqsuper}
\varepsilon(f_2f_1;g_1,\dots,g_N)=\sum_{0\le m\le N}(-1)^{|f_1|(|g_{m+1}|+\dots+|g_N|+N-m)}\varepsilon(f_2;g_{m+1},\dots,g_N)\star\varepsilon(f_1;g_1,\dots,g_m)
\end{equation}
\item[$(R_5)$]
for $f=\id_c$ and $g=\id_d$, one has $f\otimes \id_d=\id_c\otimes g\ (=\id_{c\times d})$
\end{itemize}
To make it a dg category, one should define the differential $d\varepsilon(f;g_1,\dots,g_n)$.

For $n=1$ we set:
\begin{equation}\label{eqd0}
\begin{aligned}
\ &(f\otimes \id_{d_1})\star (\id_{c_0}\otimes g)-(-1)^{|f||g|}(\id_{c_1}\otimes g)\star (f\otimes \id_{d_0})=
d\varepsilon (f;g)-\varepsilon(df;g)+(-1)^{|f|+1}\varepsilon(f;dg)
\end{aligned}
\end{equation}
For $n\ge 2$:
\begin{equation}\label{eqd1}
\begin{aligned}
\ &\varepsilon(df;g_1,\dots,g_n)=\\
&d\varepsilon({f;g_1,\dots,g_n})-\sum_{j=1}^n(-1)^{|f|+|g_n|+\dots+|g_{j+1}|+n-j} \varepsilon({f;g_1,\dots,dg_j,\dots,g_n})\big)+(-1)^{|f|+n-1}\Big[\\
&
(-1)^{|f||g_n|+|f|}(\id_{c_1}\otimes g_n)\star \varepsilon({f;g_1,\dots,g_{n-1}})
+(-1)^{|f|+\sum_{i=2}^n(|g_i|+1)+1}\varepsilon({f;g_2,\dots,g_n})\star (\id_{c_0}\otimes g_1)+\\
&\sum_{i=1}^{n-1} (-1)^{|f|+\sum_{j=i+1}^n(|g_j|+1)  } \varepsilon({f;g_1,\dots,g_{i+1}\circ g_i,\dots,g_n})\Big]
\end{aligned}
\end{equation}

One proves that 
$d^2=0$ and that the differential agrees with relations ($R_1$)-($R_5$) above.

It is clear that the twisted tensor product $C\sotimes D$ is functorial in each argument, for dg functors $C\to C^\prime$ and $D\to D^\prime$.

Note that the twisted product $C\sotimes D$ is not symmetric in $C$ and $D$.

It is {\it not} true in general that the dg category $C\sotimes D$ is quasi-equivalent to $C\otimes D$, or that these two dg categories are isomorphic as objects of $\Hot(\Cat_\dg(\k))$. 

Our interest in the twisted tensor product $C\sotimes D$ is explained by the following fact:

Let $C,D,E$ be three small dg categories over $\k$. 
Then there is a 3-functorial isomorphism of sets:
\begin{equation}\label{adj1}
\Phi\colon \Fun_{\dg}(C\sotimes D,E)\simeq \Fun_{\dg}(C,\Coh_\dg(D,E))
\end{equation}

See [Sh1, Theorem 2.2] for a proof.

One also has:

There is a dg functor $p_{C,D}\colon C\sotimes D\to C\otimes D$, equal to the identity on objects, and sending all $\varepsilon(f;\ g_1,\dots,g_s)$ with $s\ge 1$ to 0. See [Sh1, Cor. 2.3].

Note (though we will not be using it) that, for $C,D$ cofibrant, the dg functor $p_{C,D}$ is a quasi-equivalence, see Theorem \ref{theorem2}.

\section{\sc A  proof of Theorem \ref{theorem1}}\label{section2}
We firstly prove Theorem \ref{theorem3}. Then we deduce Theorem \ref{theorem1} by an inductive argument. 

\vspace{2mm}

{\it Proof of Theorem \ref{theorem3}:}

Denote by $f_i\colon i-1\to i$ a generator in $I_n$. Clearly any morphism $i-1\to j$ in $I_n$, $1\le i\le j\le n$ can be (uniquely) expresses as either as $\alpha\cdot f_j\circ f_{j-1}\circ\dots\circ f_i$. Then \eqref{eqsuper} expresses any morphism in $I_n\sotimes C$ of the form $\varepsilon(f; g_1,\dots,g_N)$ as a composition of the morphisms $\varepsilon(f_i;\  g_1,\dots,g_n)$ for the generators $f_i$. Moreover, the underlying $\k$-graded dg category of $I_n\sotimes C$ is generated by the morphisms of the following 3 types: $f_i\otimes \id_c$, $\id_i\otimes g$, $\varepsilon(f_i;\ g_1,\dots,g_n)$. 

Denote by $\Bar^{(f_i)}(C)$ the complex generated by the morphisms 
$(\id_i\otimes g_{\ell+1})*\varepsilon(f_i;\ g_1,\dots,g_\ell)*(\id_{i-1}\otimes g_0)$, $\ell\ge 0$ (where for $\ell=0$ we adapt the convention $\varepsilon(f_i;)=f_i\otimes\id$). As a complex, it is isomorphic to the bar-complex of $C$ (which is a free bimodule resolution of $C$). Then the complex of morphisms 
$(I_n\sotimes C)((i,*),(j,*))$ becomes the total complex of the tensor product
$$
K_{i,j}=\Bar^{(f_j)}(C)\otimes_C \Bar^{(f_{j-1})}(C)\otimes_C\dots\otimes_C\Bar^{(f_i)}(C)
$$
The complex $K_{i,j}$ is acyclic in bar-degrees not equal to 0, because $\Bar(C)$ is acyclic and is projective as a left (corresp., right) $C$-module. 
Thus, the cohomology of $K_{i,j}$ is isomorphic to $I_n(i,j)\otimes C(*,*)$. Theorem \ref{theorem3} is proven. 
\qed

\vspace{2mm}

To deduce Theorem \ref{theorem1}, we need
\begin{lemma}\label{lemma1}
Let $F\colon C\to D$ be a quasi-equivalence of dg categories defined over $\k$. Then $(\Id\sotimes P)\colon I_n\sotimes C\to I_n\sotimes D$ is a quasi-equivalence.
\end{lemma}
\begin{proof}
Consider the commutative diagram
$$
\xymatrix{
I_n\sotimes C\ar[r]^{\Id\sotimes F}\ar[d]_{p_C}&I_n\sotimes D\ar[d]^{p_D}\\
I_n\otimes C\ar[r]^{\Id\otimes F}&I_n\otimes D
}
$$
In this diagram, $p_C$ and $p_D$ are quasi-equivalences, by Theorem \ref{theorem3}. The lower horizontal arrow is a quasi-equivalence by elementary reasoning. Therefore, the upper horizontal arrow also is. 
\end{proof}
Theorem \ref{theorem1} follows from Theorem \ref{theorem3} and Lemma \ref{lemma1} by an elementary induction argument. 

\qed

\bigskip

{\small
\noindent {\sc Universiteit Antwerpen, Campus Middelheim, Wiskunde en Informatica, Gebouw G\\
Middelheimlaan 1, 2020 Antwerpen, Belgi\"{e}}}
\bigskip

{\small
\noindent{\sc Laboratory of Algebraic Geometry,
National Research University Higher School of Economics,
Moscow, Russia}}

\bigskip

\noindent{{\it e-mail}: {\tt Boris.Shoikhet@uantwerpen.be}}

\end{document}